\documentclass{article}

\usepackage[
	a4paper,
	margin=25mm
]{geometry}
\usepackage{setspace}
\usepackage{sectsty} 
\usepackage{titlesec}
\usepackage{appendix}

\usepackage{amsmath, amssymb, amsthm}
\usepackage{mathtools}

\usepackage[
	setpagesize=false,
	colorlinks=true,
	linkcolor=BrickRed,
        citecolor=OliveGreen,
        urlcolor=black,
	pdfencoding=auto,
	psdextra,
]{hyperref}
\usepackage{cleveref}

\usepackage{etoolbox} 
\usepackage[shortlabels]{enumitem}
\usepackage{xparse} 

\usepackage{graphicx}
\usepackage[dvipsnames]{xcolor}
\usepackage{tikz}

\usepackage{todonotes}
\usepackage[
    deletedmarkup=sout,
    commentmarkup=uwave,
    authormarkuptext=name
]{changes}

\usepackage{comment}

\setlist[enumerate,1]{label=(\arabic*), ref=(\arabic*)}
\setlist[enumerate,3]{label=(\roman*), ref=(\roman*)}

\allsectionsfont{\boldmath} 

\titlelabel{\thetitle.\quad}

\theoremstyle{plain}
\newtheorem{theorem}{Theorem}[section]
\newtheorem{lemma}[theorem]{Lemma}

\newtheorem{proposition}[theorem]{Proposition}

\newtheorem{conjecture}[theorem]{Conjecture}

\newtheorem{question}[theorem]{Question}

\newtheorem{claim}[theorem]{Claim}

\makeatletter
\newenvironment{claimproof}[1][Proof]{\par
	\pushQED{\qed}%
	
	\normalfont \topsep6\p@\@plus6\p@\relax
	\trivlist
	\item[\hskip\labelsep
	\textit{#1}\@addpunct{.}~]\ignorespaces
}{%
	\popQED\endtrivlist\@endpefalse
}
\makeatother

\newlist{Cases}{enumerate}{3}
\setlist[Cases]{parsep=0pt plus 1pt}
\setlist[Cases,1]{wide=0pt, listparindent=\parindent,
    label = \textbf{Case~\arabic*:}, ref = \arabic*}
\setlist[Cases,2]{wide=\parindent, listparindent=\parindent,
    label = \textbf{Case~\arabic{Casesi}-\arabic{Casesii}:}}

\crefname{Casesi}{case}{cases}

\newcounter{case}
\AtBeginEnvironment{proof}{\setcounter{case}{0}}

\crefname{case}{case}{cases}

\theoremstyle{definition}
\newtheorem{definition}[theorem]{Definition}

\def\A{\mathcal{A}}
\def\B{\mathcal{B}}

\newcommand{\calA}{\mathcal{A}}
\newcommand{\calB}{\mathcal{B}}

\newcommand{\calH}{\mathcal{H}}

\newcommand{\calP}{\mathcal{P}}

\newcommand{\ord}{\mathrm{Ord}}
\newcommand{\kernel}{\mathrm{Ker}}

\newcommand*{\ceilfrac}[2]{\mathopen{}\left\lceil\frac{#1}{#2}\right\rceil\mathclose{}}
\newcommand*{\floorfrac}[2]{\mathopen{}\left\lfloor\frac{#1}{#2}\right\rfloor\mathclose{}}

\newcommand*{\abs}[1]{\lvert #1\rvert}

\makeatletter
\NewDocumentCommand{\xsideset}{mmme{_^}}{%
  \mathop{%
    \settowidth{\dimen0}{$\m@th\displaystyle#3$}%
    \dimen0=.5\dimen0
    \settowidth{\dimen2}{$%
      \m@th\displaystyle#3%
      \IfValueT{#4}{_{#4}}%
      \IfValueT{#5}{^{#5}}%
    $}%
    \dimen2=.5\dimen2
    \advance\dimen2 -\dimen0
    \sbox6{\scriptspace\z@$\displaystyle{\vphantom{#3}}#1$}
    \sbox8{\scriptspace\z@$\displaystyle{\vphantom{#3}}#2$}
    \ifdim\wd6>\dimen2 \kern\dimexpr\wd6-\dimen2\relax\fi
    {%
     \mathop{\llap{\copy6}{\displaystyle#3}\rlap{\copy8}}\limits
     \IfValueT{#4}{_{#4}}%
     \IfValueT{#5}{^{#5}}%
    }%
    \ifdim\wd8>\dimen2 \kern\dimexpr\wd8-\dimen2\relax\fi
  }%
}
\makeatother

\let\originalleft\left
\let\originalright\right
\renewcommand{\left}{\mathopen{}\mathclose\bgroup\originalleft}
\renewcommand{\right}{\aftergroup\egroup\originalright}

\makeatletter
\newcases{lrcases*}
  {\quad}
  {$\m@th{##}$\hfil}
  {{##}\hfil}
  {\lbrace}
  {\rbrace}
\makeatother

\usetikzlibrary{matrix,arrows,decorations.pathmorphing}
\usetikzlibrary{positioning}
\usetikzlibrary{graphs} 
\usetikzlibrary{graphs.standard}

\title{On the order of intersecting hypergraphs}
\author{Stijn Cambie\thanks{Department of Computer Science, KU Leuven Campus Kulak-Kortrijk, 8500 Kortrijk, Belgium. Supported by the Research Foundation Flanders (FWO) (grant number 1225224N). E-mail: {\tt stijn.cambie@hotmail.com}} \and Jaehoon Kim\thanks{Department of Mathematical Sciences, KAIST, South Korea.
         Supported by the National Research Foundation of Korea (NRF) grant funded by the Korea government(MSIT) No. RS-2023-00210430. Email: {\ttfamily $\{$jaehoon.kim, hyunwoo.lee$\}$@kaist.ac.kr}}
\and Hyunwoo Lee\footnotemark[2]~\thanks{Extremal Combinatorics and Probability Group (ECOPRO), Institute for Basic Science (IBS), Daejeon, South Korea. Supported by the Institute for Basic Science (IBS-R029-C4). Email: \texttt{hongliu@ibs.re.kr}. }
\and 
Hong Liu\footnotemark[3]
\and Tuan Tran\thanks{School of Mathematical Sciences, University of Science and Technology of China, Hefei, Anhui, China. Supported by the National Key Research and Development Program of China 2023YFA1010201 and Excellent Young Talents Program (Overseas) of the National Natural Science Foundation of China. Email: \texttt{trantuan@ustc.edu.cn}.}
}

\usepackage[
    backend=biber,
    sorting=nyt,
    maxbibnames=99
]{biblatex}
\begin{document}
\maketitle

\begin{abstract}
    Determining the maximum number of edges in an intersecting hypergraph on a fixed ground set under additional constraints is one of the central topics in extremal combinatorics. In contrast, there are few results on analogous problems concerning the maximum order of such hypergraphs.
    In this paper, we systematically study these vertex analogues.
\end{abstract}


\section{Introduction}\label{sec:intro}

A central topic in extremal set theory is the determination of possible sizes of hypergraphs that satisfy certain structural constraints. One widely studied constraint is the intersecting property, where a hypergraph is called {\em intersecting} if every pair of distinct edges shares at least one common vertex.

A famous theorem by Erd\H{o}s, Ko and Rado~\cite{erdos-ko-rado} states that any intersecting $k$-uniform hypergraph ($k$-graph) on $n \geq 2k$ vertices has at most $\binom{n-1}{k-1}$ edges, which is best possible. Inspired by the Erd\H{o}s-Ko-Rado theorem, there have been extensive studies of intersecting uniform hypergraphs. 
Notable results include the complete intersection theorem due to Ahlswede and Khachatrian~\cite{Ahlswede-Khachatrian-I,Ahlswede-Khachatrian-II}, and bounds on the transversal number of intersecting $k$-graphs 
by Erd\H{o}s and Lov\'{a}sz~\cite{Erdos-Lovasz}, to name a few.

The \emph{order} of a hypergraph $\calH$ is defined as the number of non-isolated vertices in $\calH$, denoted by $\ord(\calH)$. 
We note that the term order is also used when referring to projective planes, where it has a different meaning. In this paper, the intended usage will be clear from context. In the study of intersecting hypergraphs, isolated vertices are often irrelevant, making the order a more natural parameter than the total number of vertices. While numerous results have been established concerning the size of intersecting hypergraphs, relatively few have addressed their order. The aim of this paper is to develop order-analogues of several results on uniform intersecting hypergraphs. 
In particular, we focus on variants of several natural theorems.

The intersecting condition alone is not sufficient to bound the size of hypergraphs. Indeed, this can be easily seen by considering an arbitrary collection of sets that all contain a common element. However, such a trivial construction results in a vertex with high degree. This naturally leads to the question: what is the maximum possible size of an intersecting hypergraph when the maximum degree is bounded? The following theorem of F\"{u}redi answers this question.

\begin{theorem}[F\"{u}redi~\cite{Furedi}]\label{thm:furedi}
    Let $\calH$ be an intersecting $k$-graph with maximum degree $\Delta$. Then, 
    \[
    |E(\calH)| \leq \Big(k - 1 + \frac{1}{k} \Big) \Delta,
    \]
    and this bound is tight if there exists a projective plane of order $k$.
\end{theorem}

A hypergraph $\calH$ is called $\lambda$-intersecting if every pair of distinct edges share exactly $\lambda$ vertices. A trivial way to construct a $\lambda$-intersecting hypergraph of arbitrary size is to take edges whose pairwise intersections are all the same set of $\lambda$ vertices. Such a hypergraph is called a sunflower. Formally, $\calH$ is a sunflower if $e_1 \cap e_2 = \bigcap_{e \in E(\calH)} e$ for all distinct edges $e_1, e_2 \in \calH$.  The $\lambda$-intersecting condition is quite restrictive: even without additional assumptions, it places an upper bound on the size of the hypergraph, as long as it is not a sunflower. The following theorem, due to Deza, provides a tight upper bound.

\begin{theorem}[Deza~\cite{Deza}]\label{thm:deza}
    Let $\calH$ be a $\lambda$-intersecting $k$-graph that is not a sunflower. Then, 
    \[
    |E(\calH)| \leq \max \{\lambda (\lambda + 1) + 1, (k - \lambda)(k - \lambda + 1 ) + 1\}.
    \]
\end{theorem}

Since each edge contains $k$ vertices, any upper bound on the size of a hypergraph naturally yields an upper bound on its order, namely $\ord(\mathcal{H}) \leq k |E(\calH)|$. However, this bound is generally not tight, as intersecting hypergraphs often have vertices with higher degrees, leading to overlaps that reduce the overall order. In this paper, we establish asymptotically tight bounds on the order of such hypergraphs, presenting order-analogues of \Cref{thm:furedi} and \Cref{thm:deza} as follows.

\begin{theorem}\label{thm:furedi-order}
    Let $\calH$ be an intersecting $k$-graph with maximum degree at most $\Delta$. Then 
    \[
    \ord(\calH)\leq \frac{1}{4}k^2\Delta + \frac{3}{2}\binom{2k-2}{k-1}.
    \]
    Moreover, this bound is asymptotically tight for large $\Delta$.
\end{theorem}

\begin{theorem}\label{thm:deza-order}
    Let $k,\lambda\in\mathbb{N}$ with $\lambda < k$ and $\calH$ be a $\lambda$-intersecting $k$-graph that is not a sunflower. Then 
    \[
    \ord(\calH) \leq \frac{4}{27}(k-\lambda)^3 + O\big(\lambda(k-\lambda)^{2} + (k-\lambda)^{5/2}\big).
    \]
    Moreover, this bound is asymptotically tight when $\lambda=o(k)$.
\end{theorem}

Next, we consider the order of \emph{set pair systems}. 
A pair of edge-labelled hypergraphs $(\calA, \calB)$ is called a \emph{set pair system (SPS)} if $E(\calA) = \{e_1, \dots, e_m\}$ and $E(\calB) = \{f_1, \dots, f_m\}$ satisfy $\abs{e_i \cap f_j} \neq \varnothing$ if and only if $i \neq j$. The celebrated theorem of Bollob\'{a}s~\cite{Bollobas} on
SPS gives a sharp bound on the size $m$, which has found numerous applications in extremal set theory. 
A SPS $(\calA, \calB)$ is called \emph{$(a, b)$-bounded} if every edge of $\calA$ and $\calB$ has size at most $a$ and $b$, respectively. 
It is called \emph{$1$-cross-intersecting} if every intersecting pair in $E(\calA) \times E(\calB)$
shares exactly one vertex.
The following theorem is recently obtained by F\"{u}redi, Gy\'{a}rf\'{a}s, and Kir\'{a}ly~\cite{FGK23}.

\begin{theorem}[{\cite{FGK23}}]\label{thm:2nsps}
    For all $n \geq 4$, every $(2, n)$-bounded $1$-cross-intersecting $SPS$ $(\calA, \calB)$ satisfies
    \[
    |E(\calA)| = |E(\calB)| \leq \left(\floorfrac{n}{2} + 1 \right)\left(\Bigl\lceil \frac{n}{2} \Bigr\rceil + 1 \right),
    \]
    and this is the best possible.
\end{theorem}

We establish the following order-analogue of \Cref{thm:2nsps}.

\begin{theorem}\label{thm:order-sps}
    Let $(\A,\B)$ be a $(2,n)$-bounded $1$-cross-intersecting SPS.
    Then the number of non-isolated vertices is bounded by
    \[
    \frac{(n+3)^3}{27}+\frac{(n+3)}{3},
    \]
    and this is the best possible.
\end{theorem}


\section{Intersecting hypergraphs with bounded maximum degree}\label{sec:bounded-max-degree}

In this section, we prove \Cref{thm:furedi-order}. To do so, we rely on a result regarding the size of the \emph{kernel} of intersecting hypergraphs. The intersecting property imposes a strong structural restriction on the distribution of edges in a hypergraph, indicating that they cannot be evenly distributed across the vertex set. This suggests the existence of a vertex subset of bounded size that captures all the intersecting properties of a given hypergraph, motivating the following notion of a kernel.

\begin{definition}
    Let $\calH$ be an intersecting hypergraph. 
    A vertex subset $A \subseteq V(\calH)$ is called a \emph{kernel} of $\calH$ if 
    it satisfies the following conditions:
    \begin{enumerate}
        \item[$\bullet$] For all $e\in E(\calH)$, the intersection $A\cap e$ is 
        non-empty.
        \item[$\bullet$] The hypergraph $\calH'$ with $V(\calH') = A$ and $E(\calH') = \{A\cap e: e\in E(\calH)\}$ is 
        intersecting. 
    \end{enumerate}
    We denote by $\kernel(\calH)$ the minimum size of a kernel of $\calH$.
\end{definition}

It is well known that every intersecting $k$-graph has a kernel whose size depends only on $k$. The best known bound on the minimum kernel size is given by the following theorem.

\begin{theorem}[Majumder~\cite{Majumder}]\label{thm:minimum-kernel}
    Let $\calH$ be an intersecting $k$-graph. Then 
    $$\kernel (\calH) \leq \frac{3}{2}\binom{2k-2}{k-1}.$$
\end{theorem}

With this result in hand, we are now ready to prove \Cref{thm:furedi-order}.

\begin{proof}[Proof of \Cref{thm:furedi-order}]
    Let $\calH$ be an intersecting $k$-graph with maximum degree at most $\Delta$. Let $A$ be a smallest kernel of $\calH$. Then by \Cref{thm:minimum-kernel}, we have
    \begin{equation}\label{eq:1}
        |A| \leq \frac{3}{2}\binom{2k-2}{k-1}.
    \end{equation}
    
    Consider a multi-hypergraph $\calA$ such that $V(\calA) = A$ and $E(\calA) = \{e\cap A: e\in \calH\}$. By the definition of the kernel, 
    $\calA$ is 
    an intersecting hypergraph. Let $e\in E(\calA)$ be an edge of smallest size. 
    Since $\calA$ is an intersecting hypergraph with maximum degree $\Delta$, we have
    \begin{equation}\label{eq:2}
        |E(\calA)| \leq |e|\Delta.
    \end{equation}

    Note that by the minimality of $e$, every edge $f\in \calH$ contributes at most $k - |f\cap A| \leq k - |e|$ to the number of non-isolated vertices outside of $A$. Hence, 
    \begin{equation}\label{eq:3}
        \ord (\calH) \leq |A| + \sum_{f\in E(\calH)} (k - |f\cap A|) \leq |A| + (k - |e|)|E(\calH)| .
    \end{equation}

    By combining \Cref{eq:1,eq:2,eq:3}, we obtain 
    \begin{equation*}
        \ord (\calH) \leq (k - |e|)|e|\Delta + \frac{3}{2}\binom{2k-2}{k-1} \leq \frac{1}{4}k^2 \Delta + \frac{3}{2}\binom{2k-2}{k-1}. 
    \end{equation*}
    This completes the proof.
\end{proof}

The following 
result shows that the 
multiplicative factor $\frac{1}{4}$ in \Cref{thm:furedi-order} is tight.

\begin{proposition}\label{prop:furedi-order}
    Given a real number $\varepsilon > 0$, there exists $k_0 = k_0(\varepsilon) > 0$ such that for all $k \geq k_0$, there is an intersecting $k$-graph $\calH$ with maximum degree at most $\Delta$ and
    $$\ord (\calH) \geq \left(\frac{1}{4} - \varepsilon \right)k^2 \Delta - \frac{1}{4}k^3.$$ Moreover, for a prime power $q$, if $k = 2(q+1)$ and $\Delta$ is divisible by $q+1$, then there exists an intersecting $k$-graph $\calH$ with maximum degree at most $\Delta$ such that
    $$\ord(\calH) \geq \left(\frac{k^2}{4} - \frac{k}{2} + 1 \right)(\Delta + 1).$$ 
\end{proposition}

\begin{proof}
    It follows from the Prime Number Theorem that for any real $\varepsilon>0$, there exists $k_0>0$ such that for all $k>k_0$, there is a prime $q$ satisfying $\frac{1 - \varepsilon}{2} k \leq q < \frac{k}{2}$.
    Since $q$ is a prime, 
    a projective plane $\calP$ of order $q$ exists. Note that $\calP$ is an intersecting $(q+1)$-graph with $q^2 + q + 1$ edges and the degree of every vertex in $\calP$ is $q+1$.
    Consider an intersecting multi-hypergraph $\calP'$ obtained from $\calP$ by replacing each edge of $\calP$ with $ \floorfrac{\Delta}{q+1}$ multiple edges.

    For each edge $f$ of $\calP'$, introduce $k-q-1$ distinct new vertices so that the union of $f$ and these vertices forms a new $k$-uniform hyperedge. Let $\calH$ be the resulting $k$-graph. Then by the construction, $\calH$ is an intersecting $k$-graph with maximum degree at most $\Delta$. In addition, we have 
    \begin{equation}\label{eq:order-lowerbound}
        \ord (\calH) \geq (q^2 + q + 1)\floorfrac{\Delta}{q+1} (k - q - 1) + q^2 + q + 1.
    \end{equation}

    As $(1 - \varepsilon)k/2\le q\le k/2$, the inequality \eqref{eq:order-lowerbound} implies
    $$
        \ord (\calH) \geq q(k - q - 1)\Delta - (k-q-1)(q^2+q+1) \geq \left(\frac{1}{4} - \varepsilon \right)k^2 \Delta - \frac{1}{4}k^3.
    $$
Finally, if $k = 2(q+1)$ and $q+1$ divides $\Delta$, then from inequality \eqref{eq:order-lowerbound}, we obtain
    $$
        \ord (\calH) \geq \left(\frac{k^2}{4} - \frac{k}{2} + 1 \right)(\Delta + 1).
    $$
    This completes the proof.
\end{proof}


\section{$1$-intersecting hypergraphs}\label{sec:1-intersecting}

To prove \Cref{thm:deza-order}, we begin by considering the special case of $1$-intersecting $k$-graphs. This case serves as a preliminary step in the proof of \Cref{thm:deza-order}. For this initial case, we can derive a stronger result that provides bounds not only for the order but also for the kernel of $1$-intersecting $k$-graphs.

\begin{theorem}\label{thm:1-intersecting}
 Let $\calH$ be a $1$-intersecting $k$-graph that is not a sunflower. 
 Then $$\ord (\calH) + \kernel (\calH) \le \frac{4}{27}k^3 + 4k^{5/2}.$$
\end{theorem}

We first present several families of constructions that demonstrate the sharpness of the leading coefficient $\frac{4}{27}$ in \Cref{thm:1-intersecting}, as stated in the following proposition. Moreover, each family exhibits a different ratio between the order and the minimum kernel size. In particular, the family $\calH_0$ has a kernel of size $o(k^3)$ while its order is $(\frac{4}{27}-o(1))k^3$, whereas the family $\calH_{q+1}$ has the minimal kernel size $(\frac{2}{27}-o(1))k^3$, which is nearly equal to its order. 

\begin{proposition}\label{prop:order-tight}
    For all $k \ge 4$, there exist a prime number $q= (\frac{2}{3}-o(1))k$ and $1$-intersecting $k$-graphs $\calH_0,\dots, \calH_{q+1}$ which are not sunflowers and each $\calH_a$, $0\le a\le q+1$, satisfies $\kernel (\calH_a) = ak^2/9 + o(k^3)$ and $\ord (\calH_a)=(\frac 43k - a)k^2/9 + o(k^3)$, and thus
$$\ord (\calH_a) + \kernel (\calH_a) = \left(\frac{4}{27} - o_k(1)\right)k^3.$$ 
\end{proposition}

\begin{proof}
    By the Prime Number Theorem, there exists a prime $q$ such that $q=(\frac23-o(1))k\le \frac 23 (k-1)$. Consider a projective plane $\mathcal{P}$ of order $q$, and choose a line $\ell=\{v_1, \ldots, v_{q+1}\}$. Remove this line $\ell$ from the edge set of $\mathcal{P}$. Next, choose a subset $A \subset [q+1]$ of size $0 \le a \le q+1$. For each $i \in A$, duplicate the vertex $v_i$ into two vertices $\{v_i, v'_i\}$. Then, for $\left\lfloor \frac{q}{2} \right\rfloor$ of the $q$ lines (excluding $\ell$) that pass through $v_i$, replace $v_i$ with $v'_i$ on each of those lines.

Equivalently, we can begin with an affine plane of order $q$ and select $a$ parallel classes (each consisting of $q$ pairwise disjoint lines). For each of these $a$ parallel classes, split the lines evenly into two groups and add a new common vertex to each group. For the remaining $q + 1 - a$ parallel classes, add a distinct common vertex to each of the $q$ lines within the class.

At this point, we have a $(q+1)$-uniform hypergraph in which every pair of edges intersects in exactly one vertex, except for the $a \left\lfloor \frac{q^2}{4} \right\rfloor$ pairs of edges arising from the $a$ parallel classes corresponding to the set $A$.

For every pair of non-intersecting edges, we add a new vertex shared by both edges. Since each edge is disjoint from at most $\left\lceil \frac{q}{2} \right\rceil $ other edges, the number of vertices in each edge after this vertex addition becomes at most $q+1+ \left\lceil \frac{q}{2} \right\rceil \leq k$. Finally, we add new distinct vertices to each edge as needed, so that every edge contains exactly $k$ vertices. 
    This results in a non-trivial $1$-intersecting $k$-graph $\calH_a$.
    As $\calH_a$ is $1$-intersecting, its kernel must contain all vertices with degree greater than $1$, thus the kernel of $\calH_a$ contains exactly $q^2+q+1+a+a\floorfrac{q^2}4 =  a\floorfrac{q^2}{4} + O(k^2)$ vertices.
    
    On the other hand, all $aq$ edges from the chosen $a$ parallel classes contain $(k-q-1-\left\lceil \frac{q}{2}\right\rceil)$ or $(k-q-1-\left\lfloor \frac{q}{2}\right\rfloor)$ vertices of degree $1$ and the remaining $(q+1-a)q$ edges contain $(k-q-1)$ vertices of degree $1$, thus     
    the number of vertices with degree $1$ in $\mathcal{H}$ is
    $$ aq\left(k-q-1-\left\lceil \frac{q}{2}\right\rceil \right)+O(aq) + (q+1-a)q(k-q-1)=\left(\frac 23k-a\right)\frac{2k^2}{9}+o(k^3).$$
As the first term is $o(k^3)$, we can obtain  
    \begin{align*}
        \ord (\calH_a) + \kernel (\calH)&=2\kernel(\calH)+|\{v\in V(\calH): \mathrm{deg}(v) = 1\}| \\ &= 2a\floorfrac{q^2}4+(q+1-a)q(k-q-1)+ o(k^3) = \Big(\frac{4}{27} - o_k(1) \Big)k^3. \qedhere
    \end{align*}
\end{proof}

We now prove \Cref{thm:1-intersecting}.

\begin{proof}[Proof of \Cref{thm:1-intersecting}]
    Let $\calH$ be a $1$-intersecting $k$-graph that is not a sunflower. As $\calH$ is not a sunflower, there are at least $3$ edges. 
    
    We first claim that every vertex in $\calH$ has degree at most $k$.
    To see this, let $v$ be a vertex with maximum degree belonging to edges $e_1, e_2, \ldots, e_{\Delta}$. Since $\calH$ is not a sunflower, there is an edge $e$ that does not contain $v$. The intersecting property implies that $e$ intersects each of $e_1 \backslash v, e_2\backslash v, \ldots, e_{\Delta}\backslash v$. Since the latter are pairwise disjoint, we conclude that $k=\abs{e}\ge \Delta$. 

    Let $A \subseteq V(\calH)$ be the set of vertices of $\calH$ with degree at least $\sqrt{k}$. We remark that for every edge $e \in E(\calH)$, because $\calH$ is $1$-intersecting and every vertex has degree at most $k$, we have  
        \begin{equation*}
            \abs{E(\calH)}= 1+\sum_{v \in e} \left(\deg(v)-1\right) < \sum_{v \in e} \deg(v)<k|e\cap A|+k\sqrt{k}.
        \end{equation*}        
    Consequently, if there is an edge $e\in E(\calH)$ with $\abs{e \cap A }< \sqrt k$, then $\abs{E(\calH)} < 2k \sqrt{k}$ and thus $$\ord (\calH) + \kernel (\calH) \leq 2\ord (\calH)  < 2 k \abs{E(\calH)} < 4k^2 \sqrt{k} < \frac{4}{27}k^3 + 4 k^{5/2}.$$ 
    We may then assume that for every $e\in E(\calH)$, $|e\cap A|\ge\sqrt{k}$.

    Let $E(\calH) = \{e_1, \dots, e_m\}$. 
    For each $i\in [m]$, let $e'_i= e_i\cap A$ and let $\ell_i=\abs{e'_i}$. As $\ell_i\geq \sqrt{k}$ and $\calH$ is $1$-intersecting, all sets $e'_1,\dots, e'_m$ are distinct. We assume $\ell_1\leq \ell_2 \leq \dots \leq \ell_{m}$. 

    Note that for every edge $e \in E(\calH)\backslash \{e_1, e_2\}$, either it intersects both $e'_1$ and $e'_2$, or it contains a vertex in $(e_1 \cup e_2) \setminus A$.
    As $\calH$ is $1$-intersecting, depending on $\lvert e'_1 \cap e'_2 \rvert$ being $0$ or $1$, there are at most $\ell_1\ell_2$ or $(\ell_1-1)(\ell_2-1)+k$ edges in $E(\calH)\setminus \{e_1,e_2\}$ that intersect both $e'_i$ and $e'_2$ (the further computations work along the same lines, for simplicity of presentation we only work with the former case). On the other hand, since all the vertices outside of $A$ have degree less than $\sqrt{k}$ in $\calH$, the number of edges $e\in E(\calH)\backslash \{e_1, e_2\}$ that contains a vertex in $(e_1 \cup e_2)\backslash A$ is bounded above by $\sqrt{k} |(e_1 \cup e_2)\backslash A| \leq 2k\sqrt{k} -2$. These imply 
    \begin{equation*}
        m = \abs{E(\calH)} \leq 2 + \ell_1 \ell_2 + (2k \sqrt{k} - 2) = \ell_1 \ell_2 + 2k\sqrt{k}.
    \end{equation*}
In a $1$-intersecting hypergraph, the minimum kernel is exactly the set of all vertices of degree more than $1$. Hence, $\ord (\calH) + \kernel (\calH)$ counts all degree $1$ vertices once and other vertices twice. As every vertex in $A$ has degree more than $1$, we have
  \begin{align}\label{eq:1-intersecting-upper}
       \ord (\calH) + \kernel (\calH) & \leq 2|A| + \sum_{v\notin A} d(v) \leq 2|A|+ \sum_{i\in [m]}(k-\ell_i) \nonumber \\  &\leq 2|A|+ (m-1)(k-\ell_2) + (k - \ell_1)  \le 2|A|+  \ell_2^2 (k - \ell_2) + 2k^{5/2} .
       \end{align}

    In order to finish the proof, we need to obtain an upper bound for the size of $A$. By the definition of $A$, every vertex of $A$ has degree at least $\sqrt{k}$, thus 
    $$\sqrt{k}|A|\leq \sum_{v\in A} d(v) \leq k |E(\calH)|\leq k^3,$$
    where the final inequality follows from the fact that $|E(\calH)|\leq \sum_{v\in e} d(v) \leq k^2$ for any edge $e\in E(\calH)$ in the $1$-intersecting hypergraph $\calH$. Thus, we have $|A|\leq k^{5/2}$. This together with \eqref{eq:1-intersecting-upper} implies the following
    $$
        \ord (\calH) + \kernel (\calH) \leq \ell_2^2 (k - \ell_2) + 4k^{5/2} \leq \frac{4}{27}k^3 + 4k^{5/2}.
    $$
    This completes the proof.
\end{proof} 


\section{$\lambda$-intersecting hypergraphs}\label{sec:lambda-intersecting}

In this section, we prove \Cref{thm:deza-order}. 
Before discussing the proof of \Cref{thm:deza-order}, we first provide a construction of $\lambda$-intersecting $k$-graphs demonstrating that \Cref{thm:deza-order} is tight up to $o(k^3)$ term whenever $\lambda = o(k)$. 

\begin{proposition}\label{prop:first-construction}
    For a given $\lambda = k-\Omega(k)$, there exists a $\lambda$-intersecting $k$-graph $\calH$ which is not a sunflower such that
    $$
        \ord (\calH) \geq \left(\frac{4}{27} - o_k(1)\right)(k-\lambda)^3.
    $$
\end{proposition}

\begin{proof}
    By \Cref{prop:order-tight}, there exists a $1$-intersecting $(k-\lambda+1)$-graph $\calH'$ which is not a sunflower such that
    \begin{equation}\label{eq:H'}
        \ord (\calH') \geq \left(\frac{4}{27} - o_k(1) \right)(k-\lambda)^3.
    \end{equation}
    Now we introduce $\lambda - 1$ new vertices and replace each edge $e'$ of $\calH'$ with the union of $e'$ and all the newly introduced vertices. Let $\calH$ be a hypergraph that is obtained from this. Then $\calH$ is a $\lambda$-intersecting $k$-graph. Since $\calH'$ is not a sunflower, neither is $\calH$. Hence, by \eqref{eq:H'}, we have
    $$
        \ord (\calH)=\ord (\calH')+\lambda-1 \geq \left(\frac{4}{27} - o_k(1) \right)(k-\lambda)^3.
    $$ This proves the proposition.
\end{proof}

In order to prove \Cref{thm:deza-order}, we need to collect several structural properties on $\lambda$-intersecting $k$-graphs.
For convenience, throughout this section, we denote 
$$\mu:=k - \lambda.$$ 
We first observe the following simple lemma.

\begin{lemma}\label{lem:edge-vertex}
    Let $\calH$ be a $\lambda$-intersecting $k$-graph with $m$ edges. Then $\ord(\calH) \leq \mu m + \lambda$.
\end{lemma}

\begin{proof}
    Choose an edge $e\in E(\calH)$. Since $\calH$ is a $\lambda$-intersecting hypergraph, all other edges $e'$ intersect $e$ at exactly $\lambda$ vertices. Thus, each edge of $\calH\setminus\{e\}$ contributes at most $k-\lambda$ new vertices. Thus, we have $\ord(\calH) \leq |e| + (k - \lambda)(m-1) = \mu m + \lambda$ as desired.  
\end{proof}

The following lemma proved by McCarthy and Vanstone~\cite{McCarty-Vanstone} provides useful information regarding the degree of vertices in $\lambda$-intersecting $k$-graphs.

\begin{lemma}[McCarty and Vanstone~\cite{McCarty-Vanstone}]\label{lem:degree-condition}
    Let $\calH$ be a $\lambda$-intersecting $k$-graph with $m$ edges. Then, for any $v\in V(\calH)$,
    \[
    (d(v)-\mu)(\mu+\lambda m)-\lambda d(v)^2 \le 0.
    \]
\end{lemma}

As a direct consequence of \Cref{lem:degree-condition}, one can deduce the following lemma.

\begin{lemma}\label{lem:degree-separation}
    Let $\calH$ be a $\lambda$-intersecting $k$-graph with $m \geq 20\mu$ edges. Then, for every $v\in V(\calH)$, we have either $d(v) \leq 1.1 \mu$ or $d(v) \geq m - 1.1 \mu$. 
\end{lemma}

\begin{proof}
Let $f(d)=(d-\mu)(\mu+\lambda m)-\lambda d^2$. From \Cref{lem:degree-condition}, we know $f(d(v))\le 0$.
On the other hand, since $m\ge 20\mu$,
\[
f(1.1\mu )=0.1\mu (\mu+\lambda m)-\lambda (1.1\mu)^2 \ge 0.1\mu \cdot 20\mu \lambda-\lambda (1.1\mu)^2>0.
\]
The discriminant of $f$ is 
\[
\Delta=(\mu+\lambda m)^2-4\lambda \mu (\mu +\lambda m)=(\mu +\lambda m)\left(\mu+\lambda (m-4\mu)\right).
\]
For $m\ge 20\mu$, we have $\Delta>0$, implying that the quadratic equation $f(d)=0$ has two distinct roots, $d_1$ and $d_2$, with $d_1<d_2$. By Vieta's formulas, $d_1+d_2=m+\frac{\mu}{\lambda}\ge m$.
For $m\ge 20\mu$, we have $\frac12 (d_1+d_2) \ge \frac{m}{2}>1.1\mu$.
Thus, since $f(0)<0<f(1.1 \mu)$,  $d_1\le 1.1\mu$ and $d_2\ge m-d_1\ge m-1.1\mu$. Since $f(d(v))\le 0$ and the coefficient of $d^2$ in $f(d)$ is negative, we must have either $d(v)\le d_1 \le 1.1\mu$ or $d(v)\ge d_2\ge m-1.1\mu$.
\end{proof}

From \Cref{lem:degree-separation}, every vertex of a $\lambda$-intersecting $k$-graph with many edges has either a high degree or a low degree. We call vertices with such high degree \emph{heavy}.

\begin{definition}
    Let $\calH$ be a $\lambda$-intersecting $k$-graph with $m \geq 20\mu$ edges. We say a vertex $v\in V(\calH)$ is \emph{heavy} if $d(v) \geq m - 1.1 \mu$.
\end{definition}

The following lemma will be frequently used in the proof of \Cref{thm:deza-order}. It says that if a $\lambda$-intersecting $k$-graph has a few numbers of heavy vertices, then the number of edges cannot be large. 

\begin{lemma}\label{lem:heavy-edge}
    Let $\calH$ be a $\lambda$-intersecting $k$-graph with $m \geq 20\mu$ edges that is not a sunflower. If $\calH$ has less than $\lambda$ heavy vertices, then $m \leq 2\mu^2$. Otherwise, $m \leq 3 \lambda \mu$.
\end{lemma}

\begin{proof}
    Assume $\calH$ has $t$ heavy vertices. We first consider the case that $t < \lambda$. Fix an edge $e\in E(\calH)$ containing $t'\le t$ heavy vertices. Note that every edge $e' \in E(\calH)$ intersects at least $\lambda - t'$ non-heavy vertices of $e$. Thus, writing $L$ for the set of non-heavy vertices, we have
    $$m(\lambda -t')\leq \sum_{e'\in E(\calH)} |e\cap e'\cap L| \leq \sum_{v\in e\cap L} d(v) \leq 1.1\mu (k-t'),$$
    which implies that (using $m \ge 20 \mu$)
    $$m \leq \frac{1.1\mu (k - t')}{\lambda - t'} \leq 1.1\mu(k-\lambda + 1) \leq 2 \mu^2.$$    
    Here, the penultimate inequality holds as $t'< \lambda$. 
    We now consider the case $t \geq \lambda$. Let $v_1, \dots, v_{\lambda}$ be heavy vertices. Consider a partition $\calH = \calH_1 \cup \calH_2$ such that $E(\calH_1)$ is the set of edges of $\calH$ that contain all the vertices $v_1, \cdots, v_{\lambda}$ and $E(\calH_2) = E(\calH) \setminus E(\calH_1)$.
    For any $i\in[\lambda]$, as $v_i$ is heavy, the number of edges in $\calH_2$ missing  $v_i$ is at most $1.1\mu$. Thus, $|E(\calH_2)| \leq 1.1\lambda \mu$.
    We note that $\calH$ is a $\lambda$-intersecting hypergraph and so the $k$-graph $\calH_1$ is a sunflower. As we assumed that $\calH$ is not a sunflower, there exists at least one edge $e \in E(\calH_2)$. As $\calH$ is $\lambda$-intersecting, $e$ must intersect at least one vertex of each $e'\in E(\calH_1)$ outside of $\{v_1, \dots, v_{\lambda}\}$. This implies $|E(\calH_1)| \leq k$. Thus we have $m = |E(\calH_1)| + |E(\calH_2)| \leq 1.1 \lambda \mu + k \leq 3\lambda \mu. $
\end{proof}


\subsection{Proof of \Cref{thm:deza-order}}
We are now ready to prove \Cref{thm:deza-order}. For convenience, we denote 
$$
    f(k, \lambda) := \frac{4}{27}(k-\lambda)^3 + 4\lambda(k-\lambda)^{2} + 100(k-\lambda)^{5/2} + 1000\lambda(k-\lambda) + 30000(k-\lambda).
$$

\begin{proof}[Proof of \Cref{thm:deza-order}]
    We may assume $\lambda > 1$ as the case $\lambda = 1$ is covered by \Cref{thm:1-intersecting}.
    Let $\calH$ be a $\lambda$-intersecting $k$-graph with $m$ edges. If $m \leq 20 \mu$, then by \Cref{lem:edge-vertex}, we have $$\ord(\calH) \leq 20\mu^2 + \lambda \leq f(k, \lambda).$$ Thus, we may assume that $m \geq 20 \mu$.

    \begin{claim}\label{clm:many-heavy-vtx}
        If $\calH$ has at least $\lambda - 1$ heavy vertices, then $\ord(\calH) \leq f(k, \lambda)$.
    \end{claim}

    \begin{claimproof}
        Let $\{v_1, \dots, v_t\}$ be the set of heavy vertices of $\calH$. We separately consider two cases, $t = \lambda - 1$ and $t \geq \lambda$.
        
        Assume $t = \lambda - 1$. Let $\calH = \calH_1 \cup \calH_2$ be the partition of $\calH$ such that $E(\calH_1) = \{e\in E(\calH): \{v_1, \dots, v_t\} \subset e\}$ and $E(\calH_2) = E(\calH) \setminus E(\calH_1)$.
        As each edge in $\calH_2$ misses at least one heavy vertex, we have $|E(\calH_2)| \leq 1.1\mu(\lambda - 1)$. By \Cref{lem:edge-vertex}, we see that
        \begin{equation}\label{eq:H2}
            \ord(\calH_2) \leq 1.1\mu^2 \lambda+\lambda.
        \end{equation}
        As $\calH$ is not a sunflower, if $\calH_1$ is a sunflower, then $E(\calH_2) \neq \varnothing$. This means all edges of $\calH_1$ must intersect all edges of $\calH_2$ outside of its kernel, so $|E(\calH_1)| \leq k$. Hence, in this case, $m \leq 1.1\mu(\lambda - 1) + k \leq 3\mu \lambda$. By \Cref{lem:edge-vertex} again, we have $$\ord(\calH) \leq 4\mu^2 \lambda \leq f(k, \lambda).$$
        
        Thus, we now assume that $\calH_1$ is not a sunflower. Let $\calH_1'$ be the hypergraph which is obtained from $\calH_1$ by removing $\{v_1, \dots, v_t\}$ from all edges. Then $\calH'_1$ is a $1$-intersecting $(k - \lambda + 1)$-graph that is not a sunflower. Thus, \Cref{thm:1-intersecting} and~\eqref{eq:H2} imply that $$\ord(\calH) \leq \ord(\calH_1) + \ord(\calH_2) = \lambda - 1 + \ord(\calH_1') + \ord(\calH_2) \leq f(k, \lambda).$$ 
        
        We now deal with the case that $\calH$ has $t \geq \lambda$ heavy vertices. Then \Cref{lem:edge-vertex,lem:heavy-edge} 
        imply $$\ord(\calH) \leq \mu m + \lambda \leq 3\lambda \mu^2 + \lambda \leq f(k, \lambda).$$
        This proves the claim. 
    \end{claimproof}

    In the light of \Cref{clm:many-heavy-vtx}, we may assume $\calH$ has at most $\lambda - 2$ heavy vertices. Consider the following set of vertices.
 $$A = \{v\in V(\calH): d(v) \geq \sqrt{\mu}\}.$$
Then the following claim holds.
    \begin{claim}\label{clm:intersection-with-A}
        If there is an edge $e\in \calH$ such that $|e \cap A| < \lambda$, then $m \leq 2 \mu^{3/2}$.
    \end{claim}

    \begin{claimproof}
        Assume there is an edge $e\in E(\calH)$ such that $|e\cap A| < \lambda$. Then every other edge $e'$ of $\calH$ intersects at least $\lambda - |e\cap A|$ vertices of $e \setminus A$, all of which have degree less than $\sqrt{\mu}$. Thus, we have 
        $m(\lambda -|e\cap A|)\leq \sum_{v\in e\setminus A} d(v) \leq \sqrt{\mu}(k-|e\cap A|),$
        implying  
        $$m \leq \sqrt{\mu} \left( \frac{k - |e\cap A|}{\lambda - |e\cap A|} \right) \leq \sqrt{\mu}(k - \lambda + 1) \leq 2\mu^{3/2}.$$
        This proves the claim.
    \end{claimproof}

    If there is an edge $e\in E(\calH)$ such that $|e\cap A| < \lambda$, then by \Cref{clm:intersection-with-A} and \Cref{lem:edge-vertex}, we have
    $$\ord(\calH) \leq 2\mu^{5/2} + \lambda \leq f(k, \lambda).$$
    Thus, we may assume that for every $e\in E(\calH)$, the size of the intersection $|e\cap A|$ is at least $\lambda$. 

    Now, we consider the (multi-)hypergraph $\calA$ with $V(\calA) = A$ and $E(\calA) = \{e\cap A: e\in E(\calH)\}$. Here, we consider $E(\calA)$ as a multiset.
   As have $\sqrt{\mu} |A| \leq \sum_{v\in A} d(v) \leq k m$ and $m \leq 2 \mu^2$ by \Cref{lem:heavy-edge}, we have
    $$
        |A| \leq 2k\mu^{3/2}.
    $$

    Now, we are ready to finish the proof.
    Let $f\in E(\calA)$ be an edge that has the smallest size, say $\ell$, among all the edges of $E(\calA)$. We partition $\calH$ into $\calH_1 \cup \calH_2$ where $E(\calH_1) = \{e\in E(\calH): |e\cap f| = \lambda \}$ and $E(\calH_2) = E(\calH) \setminus E(\calH_1)$.
    Let $e'$ be an edge of $\calH$ such that $e'\cap A = f$. As $\calH$ is $\lambda$-intersecting, all edges of $\calH_2$ contain at least one vertex of $e' \setminus f = e'\setminus A$. Since $|f| \geq \lambda$, we have $|E(\calH_2)| \le |e'\setminus A|\sqrt{\mu}\leq (k - \lambda)\sqrt{\mu} = \mu^{3/2}$. By \Cref{lem:edge-vertex}, we obtain $$\ord(\calH_2) \leq \mu^{5/2} + \lambda.$$

    Assume $f$ contains $t$ heavy vertices of $\calH$.
    Since $\calH$ has at most $\lambda - 2$ heavy vertices, we have $t \leq \lambda - 2$. By the definition of $\calH_1$, all edges of $\calH_1$ contain at least $\lambda - t$ non-heavy vertices of $f$. Hence, recalling that $L$ is the set of non-heavy vertices, we have 
    $ (\lambda-t)|E(\calH_1)| \leq \sum_{v\in f\cap L} d(v) \leq 1.1\mu (\ell-t)$. Thus, 
    $$|E(\calH_1)| \leq \frac{1.1\mu (\ell - t)}{\lambda - t}\le \frac{1.1\mu (\ell - \lambda + 2)}{2}.$$ 
    As $\ell$ is chosen to be the minimum $|e\cap A|$ over all edges $e$, we obtain
    $$\ord(\calH_1) \leq |A| + (k - \ell)|E(\calH_1)| \leq  2k\mu^{3/2} +  \frac{1.1 \mu (\ell - \lambda + 2)(k - \ell)}{2} \leq 2k \mu^{3/2} + \frac{1.1}{8}\mu (\mu+2)^2,$$
    where the last inequality follows from the AM-GM inequality. Consequently, we have 
    $$\ord(\calH) \leq \ord(\calH_1) + \ord(\calH_2) \leq 2k\mu^{3/2} +  \frac{1.1}{8}\mu (\mu + 2)^2 + \mu^{5/2} + \lambda \leq f(k, \lambda).$$
    This completes the proof.
\end{proof}

\section{1-cross-intersecting set pair systems}\label{sec:SPSs}

In this section, we prove \Cref{thm:order-sps}.

\begin{proof}[Proof of \Cref{thm:order-sps}]
Note that the collection $\mathcal{A}$ is a graph. 
In the proof of~\cite{FGK23} it is proved that $\A$ forms either an odd cycle $C_{2k+1}$ or a union of trees.
    In the first case, every set $B_i$ intersects the cycle in $k$ vertices, and hence, at most $n-k$ vertices lie outside the vertex set of $\A$.
    This implies that the number of non-isolated vertices in the SPS is $(2k+1)(n-k+2) \le 2\floorfrac{(n+1)(n+2)}4.$
    
    In the second case, $\A$ is a union of trees.
    We first prove the following claim which deals with the case of a single tree.
    One can note that~\cite[Lemma 3.2]{FGK23} is a corollary of this claim.
    
    \begin{claim}\label{clm:sumVi}
        Let $T$ be a tree with $t$ edges $e_1, e_2, \ldots, e_t.$ 
        For every edge $e_i$ ( $1 \le i \le t$), let 
        $V_i$ be the set of vertices at odd distances from $e_i.$
        Then $\sum_{i=1}^t \abs{V_i} \ge \floorfrac{t^2}{2}.$
    \end{claim}
    \begin{claimproof}
        We will prove by induction that for a tree whose bipartition classes, which we call black and white, have sizes $b$ and $w$, the sum equals $b(b-1)+w(w-1).$
        This is trivially true for a single edge, i.e. for size $t=1.$
        Suppose it is proven for all values up to $t-1\ge 1$ and we have a tree of size $t.$
        Pick a pending edge $e$, which we call $e_t$, and assume its end-vertex $v$ (the leaf) is black.
        Then the considered sum for $T \backslash e$ equals $(b-2)(b-1)+w(w-1)$ by the induction hypothesis.
        The vertices at odd distances from $e$ in $T$ are exactly the $b-1$ black vertices in $T \backslash e$. 
        Considering $T$ as a tree rooted at $v$, every edge can be matched with the end-vertex further from $v$.
        For every edge $e_i$ matched with a black vertex, $v$ is at odd distance from $e_i$ and hence has to belong to the vertex set $V_i$.
        Hence $b-1$ vertex sets $V_i$ are extended with one vertex and the additional vertex set $V_t$ has size $b-1.$
        This implies that $\sum_{i=1}^t \abs{V_i}=(b-2)(b-1)+w(w-1)+2(b-1)=b(b-1)+w(w-1),$ finishing the inductive step.
        
        To finish, we need to observe that under the constraint $b, w \in \mathbb N$ and $b+w=t+1$, the minimum of $b(b-1)+w(w-1)$ is $\floorfrac{t^2}{2},$ which is immediate since the minimum is attained when $\{b,w\} = \left\{ \floorfrac{t+1}{2}, \ceilfrac {t+1}{2} \right \}.$
    \end{claimproof}
    
    Now, suppose $\A$ consists of $k$ trees
    $T_i$, $i\in[k]$, with bipartition sizes $w_i \ge b_i$ and size $t_i=w_i+b_i-1.$
    Note that the edges of $\A$ span $\sum_{i=1}^k t_i+k$ vertices.
    Furthermore, the vertices not covered by any edges of $\A$ are irrelevant to the $1$-cross-intersecting property, thus we may assume that all those vertices have degree $1$ in our SPS without decreasing the order of the given SPS.     
    We next argue that we can reduce the problem to the case when all trees are stars.

    \begin{claim}\label{clm:all-star}
        We may assume that each tree $T_i$ in $\A$, $i\in[k]$, is a star.
    \end{claim}
    \begin{claimproof}
       Consider an edge $e\in T_i$ and let $B_e$ be the corresponding $n$-set in $\B$. Recall that $B_e$ is disjoint from $e$ and intersects all other edges of $\A$ precisely once. 
       Note that there are precisely two subsets of $V(T_j)$ which intersect every edge of a tree in exactly one vertex: the two bipartition classes. Thus, $B$ contains at least $\sum_{j\neq i} b_j$ vertices from $\A\setminus T_i$.

       It is crucial to observe that a set which intersects all edges from $T_i$ in exactly one vertex, except for $e$, is precisely equal to the vertex set of vertices at odd distances from $e$. It then follows from Claim~\ref{clm:sumVi} that $\sum_{e\in E(T_i)} |B_e\cap V(T_i)| \geq \floorfrac{t^2}{2}$. Therefore, the $n$-sets $\{B_e\}_{e\in T_i}$ cover at most
    \begin{equation}\label{eq:A}
        t_i\Big(n-\sum_{j\neq i} b_j\Big)-\floorfrac{t_i^2}{2}
    \end{equation} 
    vertices outside of $\A.$ 
    
    If $T_i$ is a non-star component (i.e. diameter at least $3$), then it has size $t_i\ge 3$ and $b_i\ge 2$. Replace $T_i$ by a union of two stars $T_i',T_i''$ of size $t'_i=\floorfrac {t_i}{2}$ and $t''_i=\ceilfrac {t_i}{2}$ respectively. Let $\phi$ be a bijection from the edges of $T'_i\cup T_i''$  to the edges of $T_i$.
    Also replace $\{B_e\}_{e\in T_i}$ by $\{B'_e\}_{e\in T_i'\cup T''_i}$ where $B'_e$ again contains all black vertices in $T_j$ with $j\neq i$ and the center vertex of the odd distance vertices in $T'_i\cup T''_i$ and the center vertex of the component of $T'_i\cup T''_i$ not containing $e$, and the vertices in $B_{\phi(e)}$ outside $\A$.     
    Then, the $n$-sets corresponding to edges in $T'_i\cup T''_i$ now cover 
       \begin{equation*}
         t_i \Big(n-\sum_{j\neq i} b_j - 1\Big)- t'_i(t'_i-1) - t''_i(t''_i-1).
    \end{equation*}  
    vertices outside of $\A.$ 
As $|t'_i-t''_i|\leq 1$, this is at least \eqref{eq:A} plus $$\floorfrac{t_i^2}{2} - t'_i(t'_i-1) - t''_i(t''_i-1) -t_i \geq - \frac{(t'_i-t''_i)^2 +1}{2} \geq -1.$$ On the other hand, the number of vertices in $\A$ has increased by one, this shows that the replacement does not decrease the total order of the $(2,n)$-bounded $1$-cross-intersecting SPS $(\A, \B)$. So we can assume all of them are stars.
    \end{claimproof}

    Let $\A$ be the union of $k$ stars with sizes $t_1, t_2, \ldots, t_k.$
    Then each set $B_e$ with $e\in T_i$ contains the center of all stars $T_j$ with $j\neq i$ not containing $e$, and the $t_i-1$ vertices from odd distance in $T_i$. 
    The $(2,n)$-bounded $1$-cross-intersecting SPS $(\A, \B)$ covers at most
    $$\sum_{i=1}^k (t_i+1)+ \sum_{i=1}^k t_i(n-(k-1)-(t_i-1))=\sum_{i=1}^k (t_i(n-k-t_i+3)+1)$$ vertices.
    For fixed $k$ and total size $t$, this sum is maximized when all $t_i$ differ by at most $1.$
    Let $t=\frac{\sum_{i=1}^k t_i}{k}.$
    Then we have an upper bound which equals $tk(n-k-t+3)+k.$
    This is maximized when $t=k=\frac{n+3}{3}$, giving the desired upper bound $\frac{(n+3)^3}{27}+\frac{n+3}{3}$.
    
    When $3 \mid n$, equality can be attained.
    Here $\A$ are the edges of $\frac{n+3}{3}$ stars each of size $\frac{n+3}{3}$, say $T_1, T_2, \ldots T_{\frac{n+3}{3}}.$
    For each edge $e \in E(T_i)$ in $\A$, one can take the corresponding $n$-set in $\B$ to be the union of $(T_i \setminus e)$, the center of the other stars $T_j$ and an additional distinct set of $\frac{n}{3}$ other vertices.
\end{proof}


\section{Concluding remarks}

In this paper, we establish several results on the order of intersecting hypergraphs. Our main 
results, \Cref{thm:furedi-order,thm:deza-order}, are asymptotically tight; however, there remain error terms that our approach cannot fully eliminate.
Specifically, we believe that the additive term $\frac{3}{2}\binom{2k-2}{k-1}$ in \Cref{thm:furedi-order} can be replaced by a polynomial in $k$. This term originates from an upper bound on the minimum size of the kernel in intersecting 
$k$-graphs. On the other hand, for $\calH$ to have order close to $\frac{1}{4}k^2 \Delta$, it must resemble a projective plane of order roughly $k/2$. 
Since the kernel of a projective plane has size quadratic in its order, a more refined structural analysis of intersecting $k$-graphs may improve \Cref{thm:furedi-order}.

For non-sunflower $\lambda$-intersecting $k$-graphs, as discussed in \Cref{sec:lambda-intersecting}, \Cref{thm:deza-order} is asymptotically tight when $\lambda=o(k)$. We wonder what happens for larger $\lambda$.
For exact results, the most intriguing case is that of non-sunflower $1$-intersecting $k$-graphs. From \Cref{thm:1-intersecting} and \Cref{prop:order-tight}, we know that $\max_{\calH} \ord (\calH) = \left(\frac{4}{27} + o_k(1) \right)k^3$, where the maximum runs over all non-sunflower $1$-intersecting $k$-graphs $\calH$. It would be nice to determine the exact value of this maximum.

There is a similar notion of sunflower for intersecting hypergraphs, namely trivial intersecting hypergraphs, in which all edges share a common vertex.
Note that being a non-trivial intersecting $k$-graph is a stricter condition than being a non-sunflower.
The following conjecture is attributed to Hall in \cite{Chowdhury}, though we could not locate it in \cite{Hall}.

\begin{conjecture}[Hall~\cite{Hall}]\label{conj:Hall-conj}
    Let $\calH$ be a non-trivial $\lambda$-intersecting $k$-graph,
    where $k$ is sufficiently large as a function of $\lambda$. Then $|E(\calH)| \leq \frac{k(k-1)}{\lambda} + 1.$
\end{conjecture}

For $\lambda = 2$, Hall~\cite{Hall} proved that \Cref{conj:Hall-conj} holds for every $k\ge3$, and Chowdhury~\cite{Chowdhury} confirmed its validity for $\lambda = 3$. However, the general case is still widely open. 
If \Cref{conj:Hall-conj} holds, an immediate corollary is that the order of a non-trivial $\lambda$-intersecting $k$-graph can be bounded by $k + \frac{k(k-1)(k-\lambda)}{\lambda}$ using \Cref{lem:edge-vertex}, though this bound is unlikely to be sharp.

If a biplane of order $q$ does exist, then for $k=\frac 32 (q+2)$, there exists a non-trivial $2$-intersecting hypergraph of order $(1+o(1))\frac{2k^3}{27}$.
This construction involves extending every edge with $k-(q + 2)$ unique vertices. However, the existence of infinitely many biplanes is still open.
Cameron~\cite{Cameron99}
notes that deciding whether or not infinitely many biplanes exist, seems far out of reach. In fact, the existence of the biplane with parameters $(121,16,2)$ remains unresolved.
More generally, if a symmetric $2$-$(v, q + \lambda, \lambda)$-design exists, setting $k = \frac{3}{2}(q + \lambda)$ yields a non-trivial $\lambda$-intersecting hypergraph of order $(1 + o(1)) \frac{4k^3}{27 \lambda}$. One may wonder if, as an analogue of Hall's conjecture, this bound is also sharp for non-trivial $\lambda$-intersecting $k$-graphs.

\begin{question}
    For fixed $\lambda$ and sufficiently large $k$, is it true that a non-trivial $\lambda$-intersecting 
    $k$-graph has order bounded by $(1+o_k(1))\frac{4k^3}{27\lambda}$?
\end{question}

\printbibliography

\end{document}